\documentclass[12pt]{article}

\usepackage{amssymb}
\usepackage{amsmath,amsthm}
\usepackage{color,colortbl}

\usepackage{tikz}
\usetikzlibrary{matrix}
\usetikzlibrary{patterns}
\usetikzlibrary{calc,fit}

\newcommand\mhline[3][]{%
  \node[fit=(#2-#3-1),inner sep=0pt,outer sep=0pt](R){};
  \foreach \i in {#3,...,#3}\node[fit=(R) (#2-#3-\i),inner sep=0pt,outer sep=0pt](R){};
  \draw[#1,dotted] (R.north -| #2.west) -- (R.north -| #2.east);
}

\newcommand\mvline[3][]{%
  \pgfmathtruncatemacro\hc{#3-1}
  \draw[#1,dotted]({$(#2-1-#3)!.5!(#2-1-\hc)$} |- #2.north) -- ({$(#2-1-#3)!.5!(#2-1-\hc)$} |- #2.south);
}
\tikzset{square matrix/.style={draw,
    matrix of nodes,
    column sep=-\pgflinewidth, row sep=-\pgflinewidth,
    nodes={
      minimum height=15pt,
      anchor=center,
      text width=13pt,
      align=center,
      inner sep=0pt,
      outer sep=0pt
    },
  },
  square matrix/.default=2cm
}
\providecommand{\km}{|[pattern=crosshatch, pattern color=red!40]|}
\providecommand{\mk}{|[fill=blue!40]|}

\newtheorem{thm}{Theorem}[section]
\newtheorem{cor}{Corollary}[section]
\newtheorem{lem}{Lemma}[section]
\newtheorem{ex}{Example}[section]
\newtheorem{conj}{Conjecture}[section]

\usepackage[margin=1.5cm]{geometry}

\providecommand{\mrk}{\cellcolor[gray]{.75}}

\renewcommand\emptyset{\varnothing}
\renewcommand{\geq}{\geqslant}
\renewcommand{\leq}{\leqslant}
\renewcommand{\ge}{\geqslant}
\renewcommand{\le}{\leqslant}

\def\eref#1{$(\ref{#1})$}

\def\lref#1{Lemma~$\ref{#1}$}
\def\tref#1{Theorem~$\ref{#1}$}
\def\fref#1{Figure~$\ref{#1}$}
\def\cjref#1{Conjecture~$\ref{#1}$}
\def\cyref#1{Corollary~$\ref{#1}$}

\def\Z{\mathbb{Z}}
\def\id{\varepsilon}
\def\sym{\mathcal{S}}

\title{Latin squares with maximal partial transversals\\ of many lengths}

\author{
Anthony B. Evans
\thanks{
 Department of Mathematics and Statistics, Wright State University, 
 Dayton, OH 45435, United States.}\\ 
\texttt{anthony.evans@wright.edu}\\ 
\and
Adam Mammoliti
\thanks{
School of Mathematics, Monash University, Clayton 3800, Australia.}\\ 
\texttt{adam.mammoliti@monash.edu}\\ 
\and
Ian M. Wanless
\footnotemark[2]\\ 
\texttt{ian.wanless@monash.edu}\\
}

\date{}

\begin{document}

\maketitle

\begin{abstract}
A {\em partial transversal} $T$ of a Latin square $L$ is a set of
entries of $L$ in which each row, column and symbol is represented at
most once.  A partial transversal is {\em maximal} if it is not
contained in a larger partial transversal.  Any maximal partial
transversal of a Latin square of order $n$ has size at least
$\lceil\frac{n}{2}\rceil$ and at most $n$. We say that a Latin square
is {\em omniversal} if it possesses a maximal partial transversal of
all plausible sizes and is {\em near-omniversal} if it possesses a
maximal partial transversal of all plausible sizes except one.

Evans (2019) showed that omniversal Latin squares of order $n$ exist for any
odd $n \neq 3$.  By extending this result, we show that an omniversal
Latin square of order $n$ exists if and only if $n\notin\{3,4\}$ and
$n \not\equiv 2 \pmod 4$.  Furthermore, we show that near-omniversal
Latin squares exist for all orders $n \equiv 2 \pmod 4$.  

Finally, we show that no non-trivial finite group has an omniversal Cayley
table, and only 15 finite groups have a near-omniversal Cayley table. In
fact, as $n$ grows, Cayley tables of groups of order $n$ miss a
constant fraction of the plausible sizes of maximal partial
transversals. In the course of proving this, we partially solve
the following interesting problem in combinatorial group
theory. Suppose that we have two finite subsets $R,C\subseteq G$ of a 
group $G$ such that $\big|\{rc:r\in R,c\in C\}\big|=m$.  How large do $|R|$
and $|C|$ need to be (in terms of $m$) to be certain that $R\subseteq xH$
and $C\subseteq Hy$ for some subgroup $H$ of order $m$ in $G$, and
$x,y\in G$?
\end{abstract}

\section{Introduction}\label{s:intro}

A Latin square $L$ of order $n$ is an $n\times n$ matrix containing
$n$ symbols such that each row and each column contains one copy of
each symbol. Such a matrix is equivalent to a set of \emph{triples}
$(r,c,s)$ where $s$ is the symbol in cell $(r,c)$. This viewpoint will
be exploited by using set notation and terminology when dealing with
Latin squares and their submatrices. The three coordinates of a triple
will be referred to as its \emph{row}, its \emph{column} and its
\emph{symbol}, respectively.

A \emph{partial transversal} $T$ of a Latin square $L$ is a set of
triples of $L$ that contains at most one representative of each row,
column and symbol of $L$. The \emph{length} of $T$ is the number of
triples in it. A partial transversal is {\em maximal} if it is not
contained in a partial transversal of greater length. Clearly, a
maximal partial transversal $T$ of a Latin square of order $n$ has
length at most $n$. For a Latin square of order $n$ a
\emph{transversal} is a partial transversal of length $n$ and a
\emph{near-transversal} is a partial transversal of length $n-1$. For
a survey of results and applications of transversals and partial
transversals, see \cite{Wan11}.  The following folklore result
characterises all possible lengths that a maximal partial transversal
may have.

\begin{lem}\label{l:maxbounds}
If $\ell$ is the length of a maximal partial transversal of a Latin
square of order $n$, then
\begin{equation}\label{e:posslen}
\bigg{\lceil}\frac{n}{2}\bigg{\rceil}\le \ell \le n.
\end{equation}
\end{lem}

Best \emph{et al.} \cite[Thm~12]{BMSW19} showed for $n\ge5$ that all
values of $\ell$ consistent with \eref{e:posslen} are achieved,
provided the host Latin square is allowed to depend on $n$ and $\ell$.
In this paper we explore a stronger condition where a fixed Latin
square achieves maximal partial transversals for all lengths
satisfying~\eref{e:posslen}. More generally, we will consider Latin
squares that possess maximal partial transversals of many different
lengths.  To do so, we make the following definitions.  A Latin square
of order $n$ will be called \emph{omniversal} if it has maximal
partial transversals of each length satisfying~\eref{e:posslen}
and will be called \emph{near-omniversal} if it has maximal partial
transversals of all but one of the lengths that satisfy~\eref{e:posslen}.

By constructing a special class of Latin squares,  
Evans~\cite{Eva19} proved the following. 

\begin{thm}\label{t:panoddn}
There exists an omniversal Latin square of order $n$ for all odd $n \neq 3$.
\end{thm}

By considering Latin squares of even order, we will determine the
necessary and sufficient conditions for the existence of an omniversal
Latin square of order $n$.

\begin{thm}\label{t:panall}
There exists an omniversal Latin square of order $n$ if and only if 
$n\notin\{3,4\}$ and $n \not\equiv 2 \pmod 4$.
\end{thm}

One can easily see that every Latin square of order 3 is
near-omniversal and that no omniversal or near-omniversal Latin
squares of order $4$ exist. All Latin squares of order 4 are equivalent
to the Cayley table of a group and $\Z_2\times\Z_2$ has only
transversals, while $\Z_4$ has only maximal near-transversals.
For the remaining orders excluded in \tref{t:panall}, we can at
least show that there is a near-omniversal Latin square of order $n$.

\begin{thm}\label{t:nearpan2mod4}
For any $n \equiv 2 \pmod 4$, there exists a near-omniversal Latin
square of order $n$.
\end{thm}

We also consider Cayley tables of finite groups. We show that as the
order grows, Cayley tables miss at least a constant fraction of the
lengths permitted by \lref{l:maxbounds}. As a consequence, only the
Cayley table of the trivial group is omniversal. We also obtain a
complete catalogue of near-omniversal Cayley tables, where $\Z_n$ and
$D_n$ denote the cyclic and dihedral groups of order $n$,
respectively:

\begin{thm}\label{t:nogrptran}
No non-trivial group of finite order has an omniversal Cayley table.
The only near-omniversal Cayley tables are those of $\Z_2$, $\Z_3$,
$\Z_5$, $\Z_6$, $D_6$, $D_8$ and the non-abelian groups of order $16$.
\end{thm}

Evidently, there is no hope of finding omniversal Cayley tables of
finite groups. However, our strategy for constructing omniversal Latin
squares will be to modify a Cayley table very slightly. Specifically,
we will ``turn an intercalate'', meaning that we replace a $2\times 2$
subsquare by the other possible subsquare on the same symbols.
Note that by a \emph{subsquare}, we mean a submatrix that is itself
a Latin square. One other piece of terminology we will need is this:
We say that two Latin squares are \emph{isotopic} if one can be obtained
from the other by applying three possibly distinct bijections to
the rows, columns and symbols respectively.

The paper is organised as follows. We will prove \tref{t:panall} in
Section~\ref{sec:proofpanall}, by considering the cases when $n$ is 0
and 2 modulo 4 separately.  The proof of \tref{t:nearpan2mod4} will
be given in Section~\ref{sec:nearpan2mod4}.  Section~\ref{sec:subrec}
includes preliminary results required for the proof of
\tref{t:nogrptran} but which may be of independent interest in the
study of Cayley tables of groups.  In particular, we look at how much
of a subsquare needs to be present before we know the whole subsquare
is present.  The proof of \tref{t:nogrptran} itself is presented in
Section~\ref{sec:GT}.

\section{Existence of omniversal Latin squares}\label{sec:proofpanall}

In this section we will prove \tref{t:panall}. By
\tref{t:panoddn}, we only need to consider the case when $n$
is even. We will consider the cases when $n$ is 0 and $2$ modulo 4
separately. We first show the non-existence of omniversal Latin
squares of order 2 modulo 4. We then show the existence of omniversal
Latin squares of any order $n\equiv 0 \pmod 4$, except $n=4$.

\begin{thm}\label{t:4n+2}
Let $L$ be a Latin square of order $4m+2$ for some integer $m \geq 0$.
Then either $L$ does not possess a transversal or it does not possess
a maximal partial transversal of length $2m+1$.
In particular, $L$ is not omniversal.
\end{thm}

\begin{proof}
Aiming for a contradiction, suppose that $L$ has both a maximal
partial transversal $T$ of length $2m+1$ and a transversal $T'$. Let
$A$ be the submatrix of $L$ consisting of the triples that do not
share a row or column with any element of $T$. Without loss of
generality, we may then assume that $L$ is partitioned into four
square blocks $A,B,C,D$, each of order $2m+1$ as shown:
\[
L=
\left(
\begin{array}{c|c}
  A&B  \\
  \hline
  C&D   
\end{array}
\right).
\]
As $T$ cannot be extended to a longer partial transversal, every
symbol in $A$ is contained in $T$.  It follows quickly that $A$, $B$,
$C$, and $D$ are subsquares of $L$, with $A$ and $D$ on a
symbol set $S_1$, and $B$ and $C$ on a symbol set $S_2$, where 
$S_1\cap S_2=\emptyset$.

The transversal $T'$ of $L$ contains, say,
$t$ triples from $A$. It follows that $T'$ contains $2m+1-t$ triples
from each of $B$ and $C$, and $t$ triples from $D$.  On the other
hand, $T'$ must contain each of the $2m+1$ symbols in $S_1$ exactly
once.  That means that $2t=2m+1$, an impossibility from which the
result follows.
\end{proof}

As mentioned in the introduction, no Latin square of order $4$ is
omniversal (or even near-omniversal).  There are Latin squares of
order 8 that are omniversal. For example, any Latin square obtained
from the Cayley table for the group $\Z_2^3$ by turning an
intercalate is omniversal. \fref{fig:panZ222} depicts maximal partial
transversals of such a Latin square of each possible length, shown by
the blue shaded entries.

\begin{figure}[htb]
\begin{center} 
\[
\begin{tikzpicture}
\matrix[square matrix](M){
\km2&1&\km0&3&4&5&6&7\\
1&0&3&2&5&4&7&6\\
\km0&3&\km2&1&6&7&4&5\\
3&2&1&0&7&6&5&4\\
4&5&6&7&0&1&\mk{2}&3\\
5&4&7&6&\mk{1}&0&3&2\\
6&7&4&5&2&\mk{3}&0&1\\
7&6&5&4&3&2&1&\mk{0}\\
};
\mvline{M}{5}
\mhline{M}{5}
\end{tikzpicture}
\quad 
\begin{tikzpicture}
\matrix[square matrix](M){
\mk{2}&1&0&3&4&5&6&7\\
1&0&3&2&5&\mk{4}&7&6\\
0&3&2&1&6&7&4&5\\
3&2&1&0&7&6&5&4\\
4&\mk{5}&6&7&0&1&2&3\\
5&4&7&6&\mk{1}&0&3&2\\
6&7&4&5&2&3&0&1\\
7&6&5&4&3&2&1&\mk{0}\\
};
\mvline{M}{5}
\mhline{M}{5}
\end{tikzpicture}
\quad 
\begin{tikzpicture}
\matrix[square matrix](M){
2&1&0&3&\mk{4}&5&6&7\\
1&0&3&\mk{2}&5&4&7&6\\
0&\mk{3}&2&1&6&7&4&5\\
3&2&\mk{1}&0&7&6&5&4\\
4&5&6&7&0&1&2&3\\
5&4&7&6&1&0&3&2\\
\mk{6}&7&4&5&2&3&0&1\\
7&6&5&4&3&2&1&\mk{0}\\
};
\mvline{M}{5}
\mhline{M}{5}
\end{tikzpicture}
\quad 
\]
\[
\begin{tikzpicture}
\matrix[square matrix](M){
2&1&0&3&4&5&6&7\\
1&0&3&\mk{2}&5&4&7&6\\
0&3&2&1&6&7&\mk{4}&5\\
3&2&1&0&7&\mk{6}&5&4\\
4&\mk{5}&6&7&0&1&2&3\\
5&4&\mk{7}&6&1&0&3&2\\
6&7&4&5&2&3&0&\mk{1}\\
7&6&5&4&\mk{3}&2&1&0\\
};
\mvline{M}{5}
\mhline{M}{5}
\end{tikzpicture}
\quad 
\begin{tikzpicture}
\matrix[square matrix](M){
2&\mk{1}&0&3&4&5&6&7\\
1&0&\mk{3}&2&5&4&7&6\\
0&3&2&1&6&7&4&\mk{5}\\
3&2&1&0&\mk{7}&6&5&4\\
\mk{4}&5&6&7&0&1&2&3\\
5&4&7&\mk{6}&1&0&3&2\\
6&7&4&5&2&3&\mk{0}&1\\
7&6&5&4&3&\mk{2}&1&0\\
};
\mvline{M}{5}
\mhline{M}{5}
\end{tikzpicture}
\]
\caption{Maximal partial transversals of an omniversal Latin square of order $8$.}
\label{fig:panZ222}
\end{center}
\end{figure}

We will generalise the construction in \fref{fig:panZ222} to construct
omniversal Latin squares of any order that is a multiple of 4 and at
least 8.  Let $n = 8m+4q$, where $q\in\{0,1\}$ and $m>0$.
Write $G= \Z_{2}^2\times \Z_{2m+q}$ multiplicatively 
as $\langle x,y,z \,:\, 1= x^{2m+q} =y^2=z^{2},xy=yx,xz=zx,yz=zy \rangle$.  
Let $H = \langle x \rangle$ be the subgroup of $G$ generated by $x$.
Throughout this section, we will make frequent (implicit) use of
the fact that $G$ is a disjoint union of cosets 
$G=H\cup yH \cup zH\cup yz H$.

Let $L_{8m+4q}$ be the Cayley table for
$\Z_{2}^2\times\Z_{2m+q}$.  
Let $L^*_{8m+4q}$ be the Latin square formed from $L_{8m+4q}$ by
turning the intercalate with cells $(1,1),(1,y),(y,1)$ and $(y,y)$,
i.e.,
\[
L^*_{8m+4q} =\big( L_{8m+4q}\setminus\{(1,1,1),(1,y,y),(y,1,y),(y,y,1) \} \big)
\cup \{ (1,1,y),(1,y,1),(y,1,1),(y,y,y)\}.
\]
The Latin square repeated in \fref{fig:panZ222} is (isotopic to) $L_8^*$,
with the turned intercalate marked by red crosshatching in the first copy.
We will show that $L^*_{8m+4q}$ is omniversal, but first note the
following theorem, which we will use in the proof.
It follows from the recent proof of the Hall-Paige conjecture
(see \cite{BCCSZ20,Eva18,Wan11}).

\begin{thm}\label{t:transgrp}
Let $G$ be a finite group and $L_G$ be the Cayley table for $G$. Then the 
following are equivalent.
\begin{itemize}
\item[(i)] $L_G$ can be decomposed into disjoint transversals.
\item[(ii)] The Sylow $2$-subgroups of $G$ are trivial or non-cyclic.
\end{itemize}
\end{thm}

Throughout the proof of our next theorem, we say that a set of elements $S$
{\em avoids} a row, column, or symbol $g$ if no element of $S$ is in row $g$,
column $g$ or has symbol $g$, respectively. Similarly, a set of triples
$S$ avoids a cell $(r,c)$ if $S$ contains no triple $(r,c,s)$ for any 
symbol $s$.

\begin{thm}\label{t:specpan4divn}
The Latin square $L^*_{8m+4q}$ is omniversal for all $m>0$ and $q \in \{0,1\}$.
\end{thm}

\begin{proof}
First we show that $L^*_{8m+4q}$ has a transversal and a maximal near-transversal.  Clearly, the Sylow $2$-subgroups of 
$G=\Z_{2}^2\times \Z_{2m+q}$ are non-cyclic.  Thus by
\tref{t:transgrp}, $L_{8m+4q}$ can be decomposed into disjoint
transversals.  Therefore, since $L_{8m+4q}$ has order at least 8, at
least one transversal $T$ of $L_{8m+4q}$ will avoid the cells
$(1,1)$, $(1,y)$, $(y,1)$ and $(y,y)$. Hence, $T$ is also a transversal of
$L^*_{8m+4q}$.  On the other hand, if $T'$ is a transversal of
$L_{8m+4q}$ which hits the cell $(1,1)$ but none of $(1,y)$, $(y,1)$ and
$(y,y)$, then $T'\setminus\{(1,1,1)\}$ must be a maximal near-transversal of
$L^*_{8m+4q}$.  As $L_{8m+4q}$ can be decomposed into disjoint
transversals and any pair of the triples $(1,1,1)$, $(1,y,y)$,
$(y,1,y)$ and $(y,y,1)$ share a row, column or symbol, such a
transversal $T'$ of $L_{8m+4q}$ exists and hence $L^*_{8m+4q}$ has a
maximal near-transversal.

Now we show that $L^*_{8m}$ has a maximal partial transversal of length
$4m$.  Consider the subsquare $S$ of $L_{8m}$ with rows and columns
indexed by $zH \cup yzH$.  As $S$ avoids the cells $(1,1)$, $(1,y)$, $(y,1)$
and $(y,y)$, we see that $S$ is also a subsquare of $L^*_{8m}$.  Moreover, 
$zH\cup yzH$ is the coset $z\left\langle x,y \right\rangle$ of $G$, so
$S$ is isotopic to the Cayley table for
$\Z_{2}\times\Z_{2m}$. Hence as
$\Z_{2}\times\Z_{2m}$ has a non-cyclic Sylow
$2$-group, $S$ has a transversal $T$, by \tref{t:transgrp}. It
follows that $T$ is a partial transversal of $L^*_{8m}$ of length
$4m$.  As the rows and columns of $T$ are $zH \cup yzH$ and $T$
contains every symbol in $H \cup yH$, we can extend $T$ to a longer
partial transversal only if there is $(r,c,s) \in L^*_{8m}$
such that $r,c \in H \cup yH$ and $s \in zH \cup yzH$.  However, every
cell in rows and columns $H \cup yH$ of $L^*_{8m}$ contains a symbol
in $H \cup yH$.  Thus, $T$ is a maximal partial transversal of length
$4m$.

We also show that $L^*_{8m+4}$ has a maximal partial transversal of
length $8m+2$.  Let
\begin{align*}
T = \big\{ (x^i,x^i,x^{2i}),(x^iz,x^iyz,x^{2i}y),(x^iyz,x^iy,x^{2i}z),(x^{i}y,x^{i+1}z,x^{2i+1}yz) :\, 0\le i\le 2m\big\}.
\end{align*}
Clearly $T$ is a partial transversal of $L^*_{8m+4}$ of length $8m$
that avoids rows in $\{1,z,yz,y \}$, columns in $\{1,yz,y,xz\}$ and
symbols in $\{1,y,z,xyz\}$. Therefore, $T \cup \{(1,1,y),(y,yz,z) \}$
is a maximal partial transversal of $L^*_{8m+4}$ of length $8m+2$,
since there is no element of $L^*_{8m+4}$ in row $r \in \{ z,yz\}$ and
column $c \in \{y,xz\}$ with a symbol in $\{ 1,xyz\}$.

Given what we have shown already, to prove the theorem it suffices to
show that $L^*_{8m+4q}$ has a maximal partial transversal of each
length $4m+2q+k$ for $1-q\leq k\leq 4m+q-2$ (the upper bound for $k$
can be checked by considering the two possible values for $q$).
Our constructions for
partial transversals of these lengths will be in terms of $k$ and
auxiliary parameters $w,v$ and $j$. Precise values for these
parameters will be specified subsequently. At this stage we merely
stipulate that $j$ is an integer satisfying $0\le j\le 2m-1$ and that
$\{v,w\} = \{y,z\}$, where $x,y,z$ are the generators for $G$.

Let $U_w$ be the partial transversal of $L^*_{8m+4q}$ (and
$L_{8m+4q}$) defined by
\begin{align*}
U_w = \{ (x^i,x^i,x^{2i}):\,1\le i\le m-1+q\} &\cup 
\{(wx^{i+1},wx^{i},x^{2i+1}):\, 0\le i\le m-1\}\cup \\
\{(x^{m+q+i},wx^{m+q+i},wx^{2i+q}):\, 0\le i\le m-1\} &\cup 
\{(wx^{m+1+q+i},x^{m+q+i},wx^{2i+q+1}) :\, 0\le i\le m-1\}.
\end{align*}
For $u \in G$ let $\rho_u$ and $\gamma_u$ be the functions defined by 
\[
\rho_u(r,c,rc) = (ur,c,urc) \quad \textrm{ and } \quad 
\gamma_u(r,c,rc) = (r,cu,rcu) 
\]
for all $(r,c,rc) \in L^*_{8m+4q}$.
Also, for $u \in G\setminus H$  let $\sigma_u$ be the function defined by 
\[
\sigma_u(r,c,rc) = 
\begin{cases}
  (x^{-m}yzr,cx^{m}yz,rc) & \textrm{ if } r,c \in H\\
  (x^{m}yzr,cx^{-m}yz,rc) & \textrm{ if } r,c \in uH\\
\end{cases}
\]
for any $(r,c,rc) \in L^*_{8m+4q}$ such that $r,c \in H \cup uH$ and $rc \in H$.

By permuting rows and columns, $L_{8m+4q}$ can be written in the form 
\[
\left(
\begin{array}{c|c}
  A&B  \\
  \hline
  C&D   
\end{array}
\right),
\]
where the rows and columns of $A$ are indexed by $H \cup wH$, while
the rows and columns of $D$ are indexed by $vH \cup vwH$. The rows
and columns of $B$ and $C$ are indexed accordingly.  One can easily
check that $A$ and $D$ contain the symbols $H \cup wH$, while
$B$ and $C$ contain the symbols $vH \cup vwH$.

Let $K$ be a $j$-subset of $U_w$ whose symbols are elements of 
$H\setminus\{x^{2m}\}$. Clearly $U_w$ is contained in block $A$. 
It follows that $\gamma_v(K)$, $\rho_{vw}(K)$
and $\sigma_w(K)$ are in blocks $B$, $C$ and $D$ respectively.
We claim that 
\[
U=(U_w\setminus K) \cup \gamma_v(K) \cup \rho_{vw}(K) \cup  \sigma_w(K)
\]
is a partial transversal of $L^*_{8m+4q}$ (and $L_{8m+4q}$) of length
$4m-1+q+2j$.  Clearly $(U_w\setminus K)$, $\gamma_v(K)$, $\rho_{vw}(K)$ and
$\sigma_w(K)$ are each partial transversals.

If two elements of $U$ share a row, then they are from $(U_w\setminus K)$
and $\gamma_v(K)$, or from $\rho_{vw}(K)$ and $ \sigma_w(K)$.  The
former is not possible by definition of $\gamma_v$ since the rows of
elements of $\gamma_v(K)$ are exactly those of the elements of
$K$. The latter occurs only if there are two elements in $K$ that are
in rows $r$ and $r'$, such that $r =r'x^{m}$, yet no such elements of
$K$ exist.  If two elements of $U$ share a column, then they are from
$(U_w\setminus K) $ and $\rho_{vw}(K)$, or from $\gamma_v(K)$ and
$\sigma_w(K)$.  The former is not possible as the columns of elements
of $\gamma_v(K)$ are exactly those of the elements of $K$. The
latter occurs only if there are two elements in $K$ that are in
columns $c$ and $c'$, such that $c =c'wx^{m}$, yet no such elements of
$K$ exist.  If two elements of $U$ contain the same symbol, then they
are from $(U_w\setminus K) $ and $\sigma_w(K)$, or from $\gamma_v(K)$ and
$\rho_{vw}(K)$.  The former is impossible as the symbols appearing in
$\sigma_v(K)$ are exactly those appearing in $K$ and the latter is
not possible as the elements of $\gamma_v(K)$ only contain symbols
in $vH$, while the elements of $\rho_{vw}(K)$ contain only symbols in
$vwH$.  Therefore, $U$ is indeed a partial transversal of
$L^*_{8m+4q}$.  Finally, $U$ avoids the cells $(1,1)$, $(1,y)$, $(y,1)$
and $(y,y)$ and has length $4m-1+q+2j$.  We have thus proven our
claim about $U$.

Suppose that $j=\lfloor\frac{k+q-1}{2} \rfloor$.
Let $w = y$ if $k$ is even and $w=z$ if $k$ is odd. 
We extend $U$ as defined above to give
\[
T_{k,q} = 
\begin{cases}
U\cup\{(1,1,y),(x^{m}yz,x^{m}yz,1)\} & \textrm{ if $k$ is odd and } q=0,\\
U\cup\{(1,1,y),(x^{m+1}z,x^{m}z,1),(yz,y,z)\} & \textrm{ if $k$ is odd and }q=1,\\
U\cup\{(1,z,z),(x^myz,x^{m}yz,1),(yz,1,yz)\} & \textrm{ if $k$ is even and }q=0,\\
U\cup\{(1,1,y),(x^{m+1}y,x^{m}y,1)\} & \textrm{ if $k$ is even and }q=1.
\end{cases}
\]
Now we show that $T_{k,q}$ is a maximal partial transversal of length
$4m+2q+k$, as desired.  
It is easy to see that $T_{k,q}$ contains $4m+2q+k$ triples for each
case of $k$ and $q$. We have already shown that $U$ is a partial 
transversal; we now argue that the extra triples added to $U$ do not
share a row, column or symbol with any triple in $U$.
It is easy to see that $U_w$ avoids row, column and
symbol $1$ when $q=0$ and rows $1$ and $wx^{m+1}$, columns $1$ and
$wx^m$ and symbols $1$ and $w$ when $q=1$.  Let $g \in H \cup w H$.
Note that $U$ contains an element that is in row $g$, column $g$ or has
symbol $g$ if and only if an element of $U_w$ is in row $g$, column
$g$ or has symbol $g$ respectively. It follows, for any 
$g \in H\cup w H$, that $T_{k,q}$ contains
exactly one element in row $g$, one element in column $g$ and one
element that has symbol $g$.
By construction, $K$ avoids rows in
$\{1,x^m,x^{-m},wx^{m+1}\}$ columns in $\{1,x^m,wx^m\}$ and symbols
in $\{1,x^{2m},w\}$. It follows that 
$\gamma_v(K) \cup\rho_{vw}(K) \cup \sigma_w(K)$ 
avoids rows $vw$ and $vwx^{-m}$, columns $v$ and $vwx^m$ and symbols
$v$ and $vw$.  Thus $T_{k,q}$ is a partial transversal for each case
of $k$ and $q$.

Lastly, we show that $T_{k,q}$ is indeed maximal.  For any 
$g \in H\cup w H$, we argued above that $T_{k,q}$ contains elements 
in row $g$, column $g$ and with symbol $g$.  Thus, an element 
$(r,c,s) \in L^*_{8m+4q}$ could be added to $T_{k,q}$ to make a larger partial
transversal only if $r,c,s \in vH \cup vwH$.  For any $(r,c,s) \in
L^*_{8m+4q}$ with $r,c \in vH \cup vwH$, we see that $s=rc \notin vH
\cup vwH$, except possibly when $(r,c)$ is one of $(1,1)$, $(1,y)$,
$(y,1)$ or $(y,y)$. Even in that case, $r,c,s \in vH \cup vwH$ only if
$v=y$ and $(r,c,s)=(y,y,y)$.  By definition, $v=y$ only if $k$ is odd,
yet $(1,1,y) \in T_{k,q}$ when $k$ is odd.  Hence $T_{k,q}$ is
maximal, as claimed.
\end{proof}

We can now prove \tref{t:panall}.

\begin{proof}[Proof of \tref{t:panall}]
By Theorems~\ref{t:panoddn} and \ref{t:specpan4divn}, there exists
omniversal Latin squares of order $n$ when $n \neq 3$ is odd and
$n\geq 8$ is divisible by 4, respectively.  By \tref{t:4n+2}, no
omniversal Latin square of order 2 modulo 4 can exist.  Finally, one
can easily check that no omniversal Latin squares of orders 3 or 4
exist, concluding the proof.
\end{proof}

\section{Near-omniversal Latin squares of order 2 modulo 4}\label{sec:nearpan2mod4}

Given \tref{t:4n+2}, it is natural to ask whether any Latin square of
order $2$ modulo 4 can be near-omniversal. The present section is
devoted to proving this is indeed the case, namely, proving
\tref{t:nearpan2mod4}.  Any Latin square of order $2$ is clearly
near-omniversal. In \fref{fig:nearpanZ6} we present a near-omniversal
Latin square of order $6$ that is isotopic to the Cayley table of
$\Z_6$ after turning an intercalate. Maximal partial
transversals of this square of lengths $4$, $5$, and $6$ are shown by
the blue shading (again, the turned intercalate is indicated by red
crosshatching).

\begin{figure}[htb]
\begin{center} 
\[
\begin{tikzpicture}
  \matrix[square matrix](M){
3&2&4&\mk{1}&0&5\\
2&4&0&3&\mk{5}&1\\
4&0&2&5&1&\mk{3}\\
1&3&5&2&4&0\\
\mk{0}&5&1&4&3&2\\
5&1&3&0&2&4\\
  };
  \mvline{M}{4}
  \mhline{M}{4}
\end{tikzpicture}
\quad
\begin{tikzpicture}
  \matrix[square matrix](M){
\km3&\mk{2}&4&1&\km0&5\\
2&4&0&3&\mk{5}&1\\
4&0&2&5&1&\mk{3}\\
\mk{1}&3&5&2&4&0\\
\km0&5&1&\mk{4}&\km3&2\\
5&1&3&0&2&4\\
  };
\mvline{M}{4}
\mhline{M}{4}
\end{tikzpicture}
\quad
\begin{tikzpicture}
  \matrix[square matrix](M){
3&2&4&1&\mk{0}&5\\
2&4&0&\mk{3}&5&1\\
4&0&\mk{2}&5&1&3\\
\mk{1}&3&5&2&4&0\\
0&\mk{5}&1&4&3&2\\
5&1&3&0&2&\mk{4}\\
  };
\mvline{M}{4}
\mhline{M}{4}
\end{tikzpicture}
\]
\caption{Maximal partial transversals in a near-omniversal Latin square of order $6$.}
\label{fig:nearpanZ6}
\end{center}
\end{figure}

Let $n=4m+2$. The standard Cayley table of the group
$\Z_n$ has the symbol $i+j\pmod{n}$ in cell $(i,j)$ of the table,
for all $i,j\in \{0,1,\dots,n-1\}$. An isotopic form of the Cayley table,
which highlights the subgroup of index $2$, is the Latin square
\[
\left(
\begin{array}{c|c}
  A&B  \\
  \hline
  C&D   
\end{array}
\right),
\]
where $A$, $B$, $C$, and $D$ are Latin squares of order $2m+1$ 
with rows and columns indexed by 
$\Z_{2m+1}$ and entries 
\[
A=(a_{ij})=2(i+j)\phantom{\;+1}\pmod{4m+2},
\] 
\[
B=(b_{ij})=2(i+j)+1\pmod{4m+2},
\] 
\[
C=(c_{ij})=2(i+j)+1\pmod{4m+2},
\]
 and 
 \[
 D=(d_{ij})=2(i+j)+2\pmod{4m+2},
 \]
for all $i,j \in \{0,1,\dots,2m\}$. 
Let $M_{4m+2}^*$ denote the Latin square obtained from $M_{4m+2}$ by
changing four entries, i.e., turning an intercalate, by setting
 \[
 a_{00}=2m+1=d_{mm}\quad \text{and}\quad b_{0m}=0=c_{m0}.
 \]
The Latin square $M_6^*$ is depicted in \fref{fig:nearpanZ6}.  We will
prove that $M_{4m+2}^*$ is near-omniversal, but first we require the
following two lemmas.

\begin{lem}\label{l:entry0}
  Let $a \in \{0, \ldots,4m+1\}$.  Then the equation
  $4i+a \equiv 0\pmod{4m+2}$ has a solution in $\{1,\ldots,m\}$ if and only if
  $a\equiv 2 \pmod 4$, in which case there is a unique solution
  $i=m-\frac{a-2}{4}$.
\end{lem}

\begin{proof}
If there is an $i \in \{1,\ldots,m\}$ such that $4i \equiv-a
\pmod{4m+2}$, then $4i = 4m+2-a$, as $0 < 4i <4m+2$. Consequently, $a$
must be $2$ modulo $4$ and $i= m-\frac{a-2}{4}$.  Conversely, if
$a\equiv 2 \pmod{4}$, then $i= m-\frac{a-2}{4} \in \{1,\ldots,m\}$ and
satisfies $4i+a \equiv 0 \pmod{4m+2}$.
\end{proof}

\begin{lem}\label{l:entrysame}
The equation $4i+2 \equiv 4j \pmod{4m+2}$ has no solutions $(i,j)$
that satisfy $|i-j| \leq m-1$.
\end{lem}

\begin{proof}
The solutions to 
$4i+2 \equiv 4j \pmod{4m+2}$ are any $i$ and $j$ such that 
$4(i-j)+2\equiv 0 \pmod{4m+2}$. If $|i-j| \leq m-1$, then
\[
-4m+6 \leq  4(i-j)+2 \leq 4m-2.
\]
Therefore, $4(i-j)+2\equiv 0 \pmod{4m+2}$ only if $4(i-j)+2 = 0$, yet
this is impossible, as the left hand side is 2 modulo 4 while the
right hand side is 0 modulo 4.
\end{proof}

We will now prove that $M_{4m+2}^*$ is near-omniversal, thus proving
\tref{t:nearpan2mod4}.

\begin{thm}\label{t:Near4n+2}
  The Latin square $M_{4m+2}^*$ is near-omniversal for all $m \geq 0$. 
  For $m\ge1$, it does not possess a maximal partial transversal of 
  length $2m+1$.
\end{thm}

\begin{proof}
The theorem is clearly true for $m=0$. Hence, we will assume that $m\ge 1$ and 
construct a maximal partial transversal of each valid length except $2m+1$ in $M_{4m+2}^*$.
Note that in the case when $m=1$, these maximal partial transversals are exactly those depicted in \fref{fig:nearpanZ6}.
Let $H=\langle 2\rangle$ be the subgroup of $\Z_{4m+2}$ of
index $2$; thus $\Z_{4m+2}=H\cup (H+1)$.  

First we construct a maximal partial transversal of length $2m+2k+3$,
for each $k\in\{0,1,\dots,m-1\}$.  Let $T_{A}$ consist of the triples
$(i,i+k+1,4i+2k+2)$, for $0 \leq i\leq k$, from $A$ and $T_D$ consist of the
triples $(j+k+1,j,4j+2k+4)$, for $0 \leq j \leq k$, from $D$. 
It follows from \lref{l:entrysame} that $4i \equiv 4j+2 \pmod{4m+2}$
has no solutions for $i,j\in\{0,\ldots,k \}$ and hence the
symbols in $T_A$ and $T_D$ are distinct.  Let $T_C$ consist of the
triples $(i,i,4i+1)$ for $0 \leq i \leq k $ from $C$ and $T_{B}$ consist of
the triples $(j,j,4j+1)$ for $k+1 \leq j \leq 2m $ from $B$; the symbols in the
triples in $T_B \cup T_C$ are precisely the elements of $H+1$.  Thus, as
the only triples that are not in a row or column of a triple in
$T=T_{A}\cup T_B \cup T_C \cup T_D$ are from $C$ and no triple from $C$
has a symbol from $H \setminus \{0\}$ (and even the symbol $0$ is
unavailable given that it only occurs in the first column of $C$ and
$T_C$ includes a triple in that column), it follows that $T$ cannot be
extended to a longer partial transversal.  Therefore, $T$ is a maximal
partial transversal of length $2m+2k+3$.

Next, we construct a maximal partial transversal of length $2m+2k+2$
for each $k\in\{0,1,\dots,m-1\}$.  Let $T_C$ consist of the triples
$(i,i,4i+1)$ for $1 \leq i \leq k$ and $(m,0,0)$ from $C$ and $T_B$
consist of the triples $(j,j,4j+1)$ for $k+1 \leq j \leq 2m$ and
$(0,0,1)$ from $B$. The symbols in the triples of $T_B \cup T_C$
consist of precisely the elements of $\{0\}\cup(H+1)$ and $T_B \cup T_C$ 
has length $2m+2$. As no triple of $B \cup C$ contains a symbol
in $H\setminus \{0\}$, we see that $T_B \cup T_C$ can only be extended
to a longer partial transversal by adding triples of the submatrix
$A'$ of $A$ consisting of the triples $(i,j,2i+2j)$ such that 
$i\in\{1,\dots,k\}$ and $j\in\{k+1,\dots,2m\}$; and the submatrix $D'$
of $D$ consisting of the triples $(i,j,2i+2j+2)$ such that
$i\in\left(\{k+1,\dots,2m\}\cup\{0\}\right)\setminus\{m\}$ and 
$j\in\{1,\dots,k\}$.  There are only $k$ rows in $A'$ and $k$ columns in
$D'$, so at most $2k$ triples of $A' \cup D'$ can form a partial 
transversal. Hence, if $T'$ is a partial transversal of length
$2k$ of $A' \cup D'$ with symbols in $H \setminus \{0\}$, then
$T_B\cup T_C \cup T'$ is a maximal partial transversal of length
$2m+2k+2$.  We proceed by constructing such a partial transversal $T'$,
considering the cases when $m$ is even and odd separately.

First suppose that $m$ is even.  Let $T_{A'}$ consist of the triples
$(i,i+m,4i+2m)$, for $1 \leq i \leq k$, from $A'$.  Let $T_{D'}$ consist of
the triples $(j+m,j,4j+2m+2)$ for $1 \leq j \leq \min\{k,\frac{m}{2}-1\}$ from
$D'$ and let $T'_{D'}$ consist of the triples $(j+m+2,j,4j+2m+6)$ for
$\frac{m}{2} \leq j \leq k$ from $D'$.  Note that $T'_{D'}$ is empty
if $k \leq \frac{m}{2}-1$ and that the triple $(0,m-1,2m)$ is in
$T'_{D'}$ when $k=m-1$.
We show that no symbol in $T'=T_{A'}\cup T_{D'}\cup T'_{D'}$ is $0$.  
By \lref{l:entry0}, the
symbols in $T_{A'}$ cannot be zero, given that $m$ is even.  
Similarly, \lref{l:entry0} shows that 
$4j+2m+2 = 0$ for $j \in \{1,\ldots,k+1\}$ 
only when $j=\frac{m}{2}$. It follows that no symbol in $T_{D'} \cup T'_{D'}$
is 0.  Finally we show that $T'$ is a
partial transversal.  A symbol in $T_{A'}$ is the same as
one in $T_{D'}\cup T'_{D'}$ only if $4i \equiv 4j+2$ for some 
$i\in\{1,\ldots,k\}$ and $j \in \{1,\ldots,k+1\}$, but, by
\lref{l:entrysame}, there is no such $i$ and $j$.  Hence, the
symbols in the triples of $T_{A'}$ and $T_{D'}\cup T'_{D'}$ are
distinct, and so $T'$ is a partial transversal of length $2k$ of
$A'\cup D'$ with symbols only in $H\setminus \{0\}$.  Thus,
$M_{4m+2}^*$ has a maximal partial transversal of length $2m+2k+2$
when $m$ is even.

Now suppose that $m$ is odd.  Let $T_{A'}$ consist of the triples
$(i,i+m,4i+2m)$ for $1 \leq i \leq \min\{k,\frac{m-1}{2} \}$ and $T'_{A'}$
consist of the triples $(j,j+m+1,4j+2m+2)$ for $\frac{m+1}{2} \leq j \leq k$.
Also let $T_{D'}$ consist of the triples $(j+m,j,4j+2m+2)$ for 
$1\leq j\leq\min\{k,\frac{m-1}{2}\}$ and let $T'_{D'}$ consist of the triples
$(i+m+1,i,4i+2m+2)$ for $\frac{m+1}{2} \leq i \leq k$.  Note that $T'_{A'}$
and $T'_{D'}$ are empty if $k \leq \frac{m-1}{2}$.  First we show that
no symbol in $T'=T_{A'}\cup T'_{A} \cup T_{D'} \cup T'_{D'}$ is $0$.
\lref{l:entry0} shows that $4i+2m \not\equiv 0$ for any 
$i\in\{1,\ldots,k+1\} \setminus \{\frac{m+1}{2} \}$, so no symbol
in $T_{A'} \cup T'_{D'}$ is $0$.  
\lref{l:entry0} also shows that the symbols in $T'_{A'} \cup T_{D'}$ 
cannot be $0$, since $m$ is odd.  
Finally we show that $T'$ is a partial
transversal.  A symbol in $T'_{A'} \cup T_{D'}$ is the
same as one in $T_{A'} \cup T'_{D'}$ only if 
$4i=4j+2$ for some $i\in\{1,\ldots,k+1\}$ and $j\in\{1,\ldots,k\}$, 
yet by \lref{l:entrysame} no such $i$ and $j$ exist.  Hence,
$T'$ is a partial transversal of length $2k$
with symbols only in $H \setminus \{0\}$.  Therefore $M_{4m+2}^*$ has
a maximal partial transversal of length $2m+2k+2$ when $m$ is odd.

Thus far, we have established the existence of maximal partial
transversals of $M_{4m+2}^*$ of lengths $2m+2,\dots,4m+1$. The
existence of a maximal partial transversal of length $4m+2$, i.e., a
transversal, follows from Theorem~{2.3} in~\cite{Eva19}. Finally, by
\tref{t:4n+2}, $M_{4m+2}^*$ cannot have a maximal partial transversal
of length $2m+1$ as it has a transversal.  This completes the proof.
\end{proof}

\section{Extending partial subsquares}\label{sec:subrec}

A result from \cite{Eva19}, which was useful when constructing
omniversal Latin squares of odd order, showed that if a Cayley table
of a finite cyclic group contains all but one row of a Latin square (of order
at least~3) then it contains the whole of that square as a
subsquare. We will need a more general version of this result in our proof
of \tref{t:nogrptran}. However, we think the issue is interesting
enough to treat in a separate section. Informally, how much of a Latin
square do we need to see inside a Cayley table before we can be sure
the whole square is there?  One possible approach to making that
question concrete is to consider the Hamming distance between Cayley
tables (see e.g.~\cite{Dra92,ILY12,VW12}). However, we need a
different approach, because the application we have in mind is when we
encounter a submatrix which is surprisingly pure, in that it contains
few symbols for its size. Under what conditions can we be sure such a
submatrix is inside a subsquare on the same set of symbols?  In
particular, we will consider the following question.  If $R$ is a
submatrix of a Cayley table and $R$ contains $m$ symbols, then how
large does $R$ need to be (in terms of $m$) to guarantee that $R$ can
be extended to a subsquare of $L$ which contains only the $m$ symbols
in $R$?  We conjecture the following answer.

\begin{conj}\label{conj:extLSall}
Let $L$ be a Cayley table of a group $G$.  Suppose that $R$ is an 
$\alpha m \times \beta m$ submatrix of $L$ containing $m$ symbols, where
$\alpha > \frac{1}{2}$ and $\beta >\frac{2}{3}$.  Then $R$ is
contained in an $m \times m$ subsquare of $L$.
\end{conj}

We stress that whenever we refer, say, to a submatrix containing $m$
symbols, we mean that there are exactly $m$ different symbols (and not
more) that appear within the submatrix. Those $m$ symbols may be
repeated any number of times within the submatrix. We also note that
in this section we are not assuming that our groups are finite, but we
are looking at finite submatrices of their Cayley tables. We will
prove \cjref{conj:extLSall} is true for abelian groups, and prove a
weakened version for non-abelian groups. We will also show that the
conjecture is in some sense best possible.

Let $G$ be a group. For $X,Y \subseteq G$, and an element $g \in G$ we let 
$gX = \{gx :\, x \in X\}$, $Yg = \{yg :\, y \in Y\}$ and 
$XY = \{xy :\, x \in X,\, y \in Y \}$.  When the group $G$ is abelian,
we will use additive notation. That is, we will write $gX$ as $g+X$
and $XY$ as $X+Y$ etc.  The following examples show that 
the inequalities in \cjref{conj:extLSall} cannot be strengthened.

\begin{ex}\label{ex:largenonext}
Let $G$ be a group and $H$ be a finite normal subgroup of $G$ such that some
$g \in G \setminus H$ satisfies $g^2 \notin H$. Let $L$ be the Cayley
table for $G$ and $R$ be the $\frac{1}{2}m\times m$ submatrix of $L$ with
rows indexed by $H$ and columns indexed by $H \cup gH$, where
$m=2|H|$.  Then $R$ cannot be extended to an $m \times m$ subsquare $R'$ of
$L$. To see this, note that $R$ includes the row and column indexed by
the identity, and already includes the symbols in $H\cup gH$.
Therefore $R'$ must include the rows indexed by elements of $gH$, which
leads to it including the symbols in $g^2H$ as per
\fref{f:tight}(a). However, by assumption $g^2H\not\subseteq H\cup gH$.
\end{ex}

\begin{figure}
\[
\begin{array}{ccc}
\begin{minipage}[c]{0.15\textwidth}
$ \begin{array}{c|cc}
   &H&gH\\
\hline
   H&\mrk H&\mrk gH\\
gH&gH&g^2H\\
 \end{array}$
\end{minipage}
&&
\begin{minipage}[c]{0.25\textwidth}
$ \begin{array}{c|ccc}
   &H&g^{-1}H&gH\\
\hline
   H&\mrk H&\mrk g^{-1}H&gH\\
gH&\mrk gH&\mrk H&g^2H\\
 \end{array}$
\end{minipage}
\\
\\
(a)&\phantom{spacer}&(b)\\
\end{array}
\]
\caption{\label{f:tight}Examples showing \cjref{conj:extLSall} is best possible.
The shaded regions represent $R$.}
\end{figure}

\begin{ex}
Let $G$ be a group and $H$ be a finite normal subgroup of $G$ such that some
element $g \in G \setminus H$ satisfies $g^2,g^3\notin H$.  Let $L$ be
the Cayley table for $G$ and $R$ be the $\frac{2}{3}m \times\frac{2}{3}m$ 
submatrix with rows indexed by $H \cup gH$ and
columns indexed by $H \cup g^{-1}H$, where $m=3|H|$.  Then $R$ cannot
be extended to an $m\times m$ subsquare of $L$. The proof is similar
to the previous example and is illustrated in \fref{f:tight}(b). 
If $R$ does extend to a subsquare, then that subsquare must include
columns indexed by elements of $gH$, and hence include the symbols in
$g^2H$. However, by assumption, $g^2H\not\subseteq H\cup gH\cup g^{-1}H$.
\end{ex}

The previous two examples showed submatrices of Cayley tables that
could not be extended to subsquares of the Cayley table on the
same set of symbols.  One could also ask when can a submatrix $R$ of a
Cayley table $L$ be extended to a Latin square, with only symbols in
$R$, that is not necessarily a subsquare of $L$? Perhaps surprisingly,
this is always possible.

\begin{thm}\label{t:embedsubmat}
Let $R$ be any finite submatrix of a Cayley table. Suppose
that $R$ contains $n$ symbols. 
Then $R$ can be embedded in some Latin square of order $n$.
\end{thm}

\begin{proof}
A famous theorem of Ryser \cite{Rys51} gives necessary and sufficient
conditions for a matrix to be embeddable in a Latin square of order
$n$.  Theorem 1 from~\cite{Kemp} (see also \cite{Ham}) shows that in a Cayley table
these conditions are satisfied by any submatrix that contains $n$
symbols.
\end{proof}

Next we prove \cjref{conj:extLSall} for abelian groups. We
will require the following result from additive combinatorics.

\begin{thm}[Kneser's Theorem~\cite{Kne}]\label{t:Kneser}
Let $Z = X +Y$ for finite subsets $X,Y$ of an abelian group $G$. Then
\[
|Z| \geq |X|+|Y|-|H|,
\]
for the subgroup $H = \{ g \in G :\, g+Z=Z\}$.
\end{thm}

\begin{thm}\label{t:extendabelian}
Let $L$ be a Cayley table of an abelian group $G$.  Suppose that $R$
is an $\alpha m \times \beta m$ submatrix of $L$ containing $m$ symbols
where $\alpha > \frac{1}{2}$ and $\beta >\frac{2}{3}$.  Then $R$ is
contained in an $m \times m$ subsquare of $L$.
\end{thm}

\begin{proof}
Let $X$ and $Y$ be the indices of the rows and columns of $R$, respectively. 
Then the set $Z$ of symbols in $R$ is given by $Z=X+Y$.
Let $H$ be the subgroup  $\{ g \in G :\, g+Z=Z\}$. Then $X' =  X+H$ and $Y' = Y+H$  satisfy 
$X'+Y'=Z$. It follows that $S$, the submatrix of $L$ on rows $X'$ and columns $Y'$,
contains only symbols in $Z$. 
As $X \subseteq X'$ and $Y \subseteq Y'$, we see that $R$ is contained in $S$ and so it suffices to show that $S$ is in fact an $m \times m$ subsquare of $L$.

As $H$ is a subgroup, $X'=X' +H$ and $Y' = Y'+H$. In particular, $x +
H \subseteq X'$ and $y + H \subseteq Y'$ for all $x\in X$ and $y \in Y$. 
It follows that $X'$ and $Y'$ are each a union of cosets of $H$.
Consequently, $Z$ is also the union of cosets of $H$.  Let 
$|X'|=\alpha' |H|$ and $|Y'| = \beta'|H|$, i.e., $X'$ is the union of
$\alpha'$ cosets of $H$ and $Y'$ is the union of $\beta'$ cosets of
$H$. By \tref{t:Kneser},
\[
|Z| \geq |X'|+|Y'|-|H| = (\alpha'+\beta'-1)|H|.
\]

First suppose that $|Z| > (\alpha'+\beta'-1)|H|$.
As $Z$ is the union of cosets of $H$, if $|Z|\neq (\alpha'+\beta'-1)|H|$, 
then $|Z|\geq (\alpha'+\beta')|H|$.
Therefore,
\[
\alpha = \frac{|X|}{m} \leq \frac{|X'|}{m}  \leq 
\frac{\alpha'|H|}{(\alpha'+\beta')|H|}= 
\frac{\alpha'}{(\alpha'+\beta')}.
\]
By a similar argument, $\beta\leq\frac{\beta'}{(\alpha'+\beta')}$. 
However, then $\alpha+\beta\le1$, contradicting the assumptions of the
theorem.

Thus, $|Z| = (\alpha'+\beta'-1)|H|$.
By a similar argument to the above,
\[
\frac{1}{2} < \alpha \leq \frac{\alpha'}{\alpha'+\beta'-1} \quad \textrm{and}\quad \frac{2}{3} < \beta \leq \frac{\beta'}{\alpha'+\beta'-1}.
\]
Simplification yields the identities $\beta'-1 < \alpha'$ and
$2\alpha'-2 < \beta'$. As both $\alpha'$ and $\beta'$ are positive
integers, we must have $\alpha'\ge\beta'\ge2\alpha'-1$, from which it
follows that $\beta'=\alpha'=1$. Hence, $|X'| = |Y'| = |Z|$ and $S$ is
a $m \times m$ subsquare of $L$.
\end{proof}

We now consider general groups. Unlike the abelian case, 
an analogue to \tref{t:Kneser}
does not hold; counterexamples are given by Olson~\cite{Ols}.
Instead, we will make use of the following result from \cite{Ols}.

\begin{thm}\label{t:Olsonweaklowerbound}
Let $Z = XY$ for finite subsets $X,Y $ of a group $G$ such that $1 \in X$. Then either
\begin{itemize}
\item[(i)] $XZ = Z$ or; 
\item[(ii)] $|Z| \geq \frac{1}{2}|X|+|Y|$.
\end{itemize}
\end{thm}

Using this result we can prove a weakened form of \cjref{conj:extLSall}
holds for all groups.

\begin{thm}\label{t:extendLSnonabelian}
Let $L$ be a Cayley table of a group $G$.  Suppose that $R$ is an
$\alpha m \times \beta m$ submatrix of $L$ with $m$ symbols,
where $\frac{1}{2} < \alpha \leq \beta \leq 1$.  If
$\frac{\alpha}{2}+\beta >1$, then $R$ is contained in an $m \times m$
subsquare of $L$.
\end{thm}

\begin{proof}
Let $X$ and $Y$ be the indices of the rows and columns of $R$,
respectively, and $Z=XY$ be the symbols in $R$.  For any $g\in G$, $R$
is contained in an $m \times m$ subsquare of $L$ if and only if the
submatrix of $L$ on rows $gX$ and columns $Y$ is contained in an
$m\times m$ subsquare of~$L$.  Therefore without loss of generality,
we can assume that $1 \in X$.  Then by \tref{t:Olsonweaklowerbound},
$Z=XZ$, since
\[
\frac{1}{2}|X|+|Y| = \left(\frac{\alpha}{2}+\beta \right)m > m = |Z|.
\]
Consequently, $Z = X^i Z$ for all $i\in\Z$ and so $Z = HZ$
for $H = \langle X \rangle$.  Therefore as $1 \in X$, we have
$Y\subseteq Z$ and so $R$ is contained in the submatrix $S$ of $L$
with rows indexed by $H$ and columns indexed by $Z$, which only
contains symbols in $Z$.  So we can complete the proof by showing that
$S$ is in fact an $m \times m$ subsquare of $L$, for which it suffices
to prove that $|Z| = |H|$.  As $Z = HZ$, we see that $Z$ is the union
of (right) cosets of $H$. On the other hand, $H \supseteq X$ and
$|X|>\frac{1}{2}|Z|$.  Hence, $Z$ must be in fact be a coset of
$H$. It follows that $|H|=|Z|$, as required.
\end{proof}

\section{Finite group tables}\label{sec:GT}

In this section we consider the lengths of maximal partial transversals
in the Cayley tables of finite groups. We find that Cayley tables
get further and further from being omniversal as they get larger.
This enables us to diagnose which Cayley tables are omniversal or
near-omniversal. In particular, we prove \tref{t:nogrptran}.

First we state some preliminary results, starting with a well-known result
about subsquares of Cayley tables (for a proof, see \cite{BCW14}).

\begin{lem}\label{l:subsqCayley}
Let $S$ be a subsquare of the Cayley table of a group $G$ of order
$n$.  Then $G$ has a subgroup $H$ such that $S$ is formed by the rows
indexed by $xH$ and columns indexed by $Hy$ for some cosets $xH$
and $Hy$ of $H$.  In particular, $S$ must have order dividing $n$.
\end{lem}

In the case of abelian groups, we also have the following observation
due to Belyavskaya and Russu \cite{BR76} (see also
\cite[Lem.1]{BMSW19}). A proof of this observation is implicit in
earlier work of Paige \cite{Paige}. Paige proved that a finite abelian
group possesses a complete mapping if and only if its Sylow
$2$-subgroup is trivial or noncyclic (see also \cite[Thm~3.9]{Eva18}),
and a proper near complete mapping if and only if its Sylow
$2$-subgroup is nontrivial and cyclic (see also \cite[Cor.~3.10 or
  Thm~15.2]{Eva18}). It is well-known that the Cayley table of a
finite group has a transversal if and only if the group possesses a
complete mapping, and a maximal near-transversal if and only if the
group possesses a proper near complete mapping.

\begin{thm}\label{t:noabelpanmax}
  No Cayley table of a finite abelian group has both a transversal and
  a maximal near-transversal.
\end{thm}

By \tref{t:transgrp} we know exactly which Cayley tables of finite
groups have transversals.  This also tells us which Cayley tables of
groups of order $n$ have maximal partial transversals of length
$n/2$. Such a partial transversal can only be achieved by taking a
transversal of a subsquare of order $n/2$. Hence \tref{t:transgrp} and
\lref{l:subsqCayley} give us:

\begin{cor}\label{cy:halfn}
The Cayley table of a group of order $n$ has a maximal partial
transversal of length $n/2$ if and only if it has an index $2$
subgroup $H$ for which the Sylow $2$-subgroups are trivial or
non-cyclic.
\end{cor}

In particular, the Cayley tables of all groups of order $n\equiv2 \pmod
4$ and no groups of order $n\equiv4 \pmod 8$ will have a maximal partial
transversal of length $n/2$.

When a Cayley table of a group does not have a transversal, it will
have a maximal near-transversal, as a result of the following recent
result of Goddyn and Halasz \cite{GH20}.

\begin{thm}\label{t:grpBrualdi}
The Cayley table of any finite group has a near-transversal.
\end{thm}

Our first main result for this section shows that large groups are
very far from being omniversal. Indeed, they miss at least a constant
fraction of the lengths permitted by \lref{l:maxbounds}.

\begin{thm}\label{t:nosmallingrp}
Let $L$ be a Cayley table of a group $G$ of order $n$ and suppose that
$L$ contains a maximal partial transversal $T$ of length
$\ell<\frac{3}{5}n$. Then $n$ is even, $G$ contains a subgroup of
index $2$ and $\ell-\frac{n}{2}$ is even.
\end{thm}

\begin{proof}
Let $R$ be the submatrix of $L$ that consists of the triples that do
not share a row or column with any triple in $T$. Suppose that $R$
contains $m$ symbols. As $T$ is maximal, the symbols in $R$ must also
be symbols in $T$ and, in particular, $m\le\ell$.  As
$n-\ell>\frac{2}{3}\ell$ and $R$ is an $(n-\ell)\times(n-\ell)$
submatrix of $L$, \tref{t:extendLSnonabelian} implies that $R$ is
inside a subsquare $S$ of $L$ of order $m$ such that $n-\ell \leq m
\leq \ell$. Now $n>m\geq n-\ell>\frac{2}{5}n$. So the only way to have
$m\mid n$ as required by \lref{l:subsqCayley}, is if $n=2m$ and 
$G$ has an index 2 subgroup.  It follows that $L$ takes the form
\[
\left(
\begin{array}{c|c}
  A&B  \\
  \hline
  C&S
\end{array}
\right),
\]
where $A$, $S$ and $R$ have the same $m$-set of symbols.  Let
$w,x,y,z$ be the number of triples that $T$ includes from $A,B,C,S$
respectively.  Since $T$ contains every symbol in $R$ it must contain
exactly $m=w+z$ triples from $A\cup S$ and $\ell-m=x+y$ triples from
$B\cup C$.

We claim that $x=y$, from which it will follow that $\ell-m$ is even,
as required. Suppose on the contrary, that $x\ne y$. Assume that $x<y$
and $w\ge z$ (the cases when $x>y$ or $z>w$ work similarly). Then,
since $x+w<y+w\le m$, there must be a row $r$ of $A\cup B$ that is not
represented in $T$. Of the $m$ triples in row $r$ of $B$, only $x+y=\ell-m$
of them have a symbol which occurs in $T$, and $x+z < \frac{1}{2}\ell$ of them have a
column that is used in $T$. As $2x+y+z< \frac{3}{2}\ell-m<m$, there is 
a triple in row $r$ that can be used to extend $T$. This contradiction
of the maximality of $T$ completes the proof.
\end{proof}

In particular, \tref{t:nosmallingrp} shows that groups with no index 2
subgroup (including all groups of odd order) miss all lengths $\ell$
in the range $\frac12n\le\ell<\frac35n$. The inequality cannot be
replaced with equality,
as the following maximal partial transversal
in the Cayley table of $\Z_5$ shows:
\begin{equation}\label{e:Z5ell3}
  \begin{minipage}[c]{.15\textwidth}
    \begin{tikzpicture}
    \matrix[square matrix](M){
  0&1&2&3&4\\
  1&2&3&4&0\\
  2&3&4&\mk0&1\\
  3&4&0&1&\mk2\\
  4&0&\mk1&2&3\\
    };\end{tikzpicture}
    \end{minipage}
\end{equation}
Indeed, this example can be used to construct infinitely many examples
for which the theorem is tight. Suppose that $G$ is a group of order
$n$ which has a normal subgroup $N$ for which $G/N\cong\Z_5$ and the Cayley table of $N$
has a transversal. Then by combining transversals of the 3 blocks of
the Cayley table of $G$ that are mapped to the shaded cells in
\eref{e:Z5ell3} when we factor out $N$, we obtain a maximal partial
transversal of the Cayley table of $G$ that has length $3n/5$.

Groups with an index 2 subgroup miss every second length
within the range $\frac12n\le\ell<\frac35n$. Our next result covers
the other half of that range (and more).

\begin{thm}\label{t:everysecond}
Let $L$ be a Latin square of even order $n$ that contains a subsquare
$A$ of order $n/2$.  Suppose that $A$ has a near-transversal. Then $L$
contains a maximal partial transversal of length $\ell$ for all $\ell$
such that $\frac{n}{2}<\ell\le\frac{3}{4}n$ and $\ell-\frac{n}{2}$ is even.
\end{thm}

\begin{proof}
By assumption, $L$ takes the form
\[
\left(
\begin{array}{c|c}
  A&B  \\
  \hline
  C&D   
\end{array}
\right),
\]
where $A$ and $D$ are subsquares of order $n/2$ on some symbol set $S$.
Also, we are assuming that $A$ has a near-transversal $T$. Let
$s$ be the symbol in $S$ that does not appear in $T$.  We proceed by
constructing a maximal partial transversal of length $\frac{n}{2}+2x$
for all $1\le x \leq \frac{n}{8}$.  Any partial transversal of $D$ of
length less than $\frac{n}{4}$ can be extended greedily, so there is a
partial transversal $T'$ of $D$ of length $x$ that contains the symbol
$s$.  Let $T''$ be formed from $T'$ by adding the triples of $T$ whose
symbols are not used in $T'$. Necessarily $T''$ has
length $\frac{n}{2}$ and contains the symbols in $S$.  We proceed by
adding triples from $B$ and $C$ to $T''$ whose rows, columns and
symbols are distinct from the triples in $T''$ until there are exactly
$\frac{n}{2}+2x$ triples.  Let $T_{n/2}=T''$ and suppose we have
chosen $T_i \supseteq T''$ such that $T_i$ is a partial transversal of
length $i$ for some $i<\frac{n}{2}+2x$.  Suppose that no triple of
$T_i$ is in a particular row $r$ of $B$ (the case when no triple of
$T_i$ is in a column of $C$ is similar). As $A\cap T_i$ contains
exactly $\frac{n}{2}-x$ triples, there can be at most $x-1$ triples
from $B$ and $x$ triples from $C$ that are in $T_i$.  
Therefore there are at most $2x-1$ triples from $T_i$ that
share a symbol with a triple in row $r$ of $B$.
Also, as $D \cap T_i$ contains exactly $x$ triples, 
there are at most $2x-1$ triples from $T_i$ that share a column with a
triple in row $r$ of $B$.  Thus, there are at
least $\frac{n}{2}-4x+2 > 0 $ triples in row $r$ of $B$ that do not
share a row, column or symbol with any triple in $T_i$, any of which
can be added to $T_i$ to form a longer partial transversal $T_{i+1}$.
Continuing in this way, we can find a partial transversal
$T_{\ell}\supseteq T''$ of length $\ell = \frac{n}{2}+2x$.  The
partial transversal $T_{\ell}$ is maximal, as the only triples not in
the same row or column as a triple in $T_\ell$ lie in $D$, yet all the
symbols of the triples in $D$ are in $T''$ and hence also in $T_{\ell}$.
\end{proof}

Note that the hypothesis that $A$ has a near transversal is a weak
condition in the sense that is conjectured to hold for all Latin squares
(see \cite{Wan11}).
Together with \tref{t:grpBrualdi} and \lref{l:subsqCayley},
\tref{t:everysecond} shows that the Cayley tables of finite groups with an
index 2 subgroup have maximal partial transversals for each length
$\ell$ satisfying $\frac12n<\ell\le\frac34n$ and
$\ell\equiv\frac12n\pmod2$.

We next find all maximal partial transversal lengths in Cayley tables
of small groups.

\begin{thm}\label{t:smallgrp}
The Cayley tables of groups of order $n\le24$ possess maximal partial
transversals of all lengths $\ell$ not forbidden by \tref{t:transgrp},
\tref{t:noabelpanmax}, \cyref{cy:halfn} or \tref{t:nosmallingrp}, 
except for the following cases:
\begin{itemize}
\item Cayley tables of non-cyclic groups of order $8$ do not achieve $\ell=5$.
\item The Cayley table of $\Z_9$ does not achieve $\ell=6$.
\item Cayley tables of groups of order $n$ do not achieve $\ell$ when $(n,\ell)$ is one of
\begin{align*}
&(10,6),(11,8),(13,8),(15,10),(17,12),(19,12),(21,13),(21,14),(22,14),(23,14),(23,16).
\end{align*}
\item Cayley tables of $\Z_2\times\Z_{10}$ and the holomorph $\Z_5\rtimes\Z_4$ do not achieve $\ell=13$.
\item Cayley tables of groups of order $24$ other than $\Z_{24}$ and $\Z_3\rtimes\Z_8$
do not achieve $\ell=15$.
\end{itemize}
\end{thm}

\begin{proof}
The result was obtained by two independent computations. 
By using left and right translation we could assume when convenient that
any partial transversal contains the triple $(\id,\id,\id)$, where
$\id$ is the group identity.  With this assumption, a simple backtracking search
sufficed for all $n\le16$. Running a partial backtracking search for each group in the range
$17\le n\le24$ quickly found all lengths $\ell>\frac23(n+1)$ that are not forbidden
by \tref{t:transgrp} or \tref{t:noabelpanmax}. Also
\cyref{cy:halfn}, \tref{t:nosmallingrp} and \tref{t:everysecond} resolve the
case when $\ell<\frac35n$.

For $\frac35n\le\ell\le\frac23(n+1)$ and $17\le n\le24$ we adopted a
different strategy for finding a maximal partial transversal $T$ of
length $\ell$. We focused first on the ``complementary'' submatrix $S$
formed by rows and columns of the Cayley table that are not used in
$T$. We did not assume that $(\id,\id,\id)\in T$, but rather that $S$
uses row $\id$ and column $\id$. We knew that $S$ contains only
symbols that are in $T$. So we searched for $(n-\ell)\times(n-\ell)$
submatrices containing at most $\ell$ symbols. We did this by
considering all choices for $n-\ell$ rows that included row $\id$. For
each choice of rows, we then built up the set of columns by adding one
column at a time, starting with column $\id$. Most choices quickly
violated the bound on the number of symbols in $S$, so this was a fairly
quick search.  In some cases, such as when $G=\Z_{23}$ and $\ell=14$,
there were no viable choices for $S$ and hence we knew that this value
of $\ell$ is not achieved for that group. In other cases, there were
many viable choices for $S$. For example, when $G=\Z_{23}$ and
$\ell=16$ there are $7\times7$ submatrices with 13, 14, 15 or 16 symbols.

One quick test for viability of a candidate for the submatrix $S$ is
as follows.  If $S$ contains $m$ symbols where $n-\ell>\frac23m$, then
\tref{t:extendLSnonabelian} implies that $S$ is inside a subsquare
of order $m$. If $G$ has no subgroup of order $m$ then this
immediately contradicts \lref{l:subsqCayley}. If $G$ has a subgroup of
order $m$ and $m=n/2$ then the argument at the end of the proof of
\tref{t:nosmallingrp} shows that $\ell-m$ is even (given that
$\ell<\frac34n$). The case when $\ell-m$ is even is handled by
\tref{t:everysecond}. In practice, these considerations meant that we
never had to do further computations when $n-\ell>\frac23m$ and
$m>n/3$.

When a candidate for $S$ was identified and could not be handled by
the above arguments, we needed to look for $T$.  However, we now
knew much more about what we were looking for.  We knew exactly which
rows and columns $T$ must use (namely the ones that were not in
$S$). We also knew that $T$ must use all of the symbols that occur in
$S$. Typically this either determined the set of symbols in $T$
completely, or almost completely. With this extra information it was
usually feasible to determine whether or not $T$ existed by a
backtracking search.  The hardest case was the example that we have
already mentioned, namely $G=\Z_{23}$ and $\ell=16$. For that case, we
used one additional piece of information that is available in any
abelian group $G$. The sum (in $G$) of the symbols within any
partial transversal of the Cayley table for $G$ that uses a specified
set of rows and columns is just the sum of the indices of those rows
and columns (cf. the Delta lemma~\cite{Wan11}).
This, together with the knowledge that $T$
must contain every symbol in $S$ was enough to eliminate all
$7\times7$ submatrices of the Cayley table of $\Z_{23}$. For each
$7\times7$ submatrix there was no feasible set of symbols for $T$ to
use.

In all cases where a maximal partial transversal of a particular
length $\ell$ did not exist, our two programs agreed on the
reason. Both found the same number (possibly zero) of candidates for
$S$. For each candidate they agreed on whether there was a viable set
of symbols that might be used in $T$. And if there was, they agreed
that there was in fact no $T$.
\end{proof}

We can now prove \tref{t:nogrptran}. We will say that a group is
omniversal or near-omniversal if its Cayley table is omniversal or
near-omniversal, respectively.

\begin{proof} [Proof of \tref{t:nogrptran}.]
First suppose that $G$ is an omniversal group of order $n\ge2$.
By \tref{t:noabelpanmax}, $G$ must be non-abelian and by 
\tref{t:4n+2}, we know that $n\not\equiv2\pmod4$.
Also the non-abelian groups of order 8 are not
omniversal by \tref{t:smallgrp}. Our observations thus far
ensure that $n\ge12$. In that case, \tref{t:nosmallingrp}
shows that the Cayley table of $G$ has no maximal partial transversal of length
$\lfloor n/2\rfloor+1$, which completes the proof that there
are no non-trivial omniversal groups.

Next we suppose that $G$ is a near-omniversal group of order $n\ge2$.
From \eref{e:Z5ell3}, \tref{t:noabelpanmax} and \tref{t:transgrp}
it follows that $\Z_3$ and $\Z_5$ are near-omniversal. For
odd $n>5$, \tref{t:nosmallingrp} shows that $G$ misses $\ell=(n+1)/2$.
Together with \tref{t:noabelpanmax}, this eliminates the remaining
abelian groups of odd order. Non-abelian groups of odd order
are eliminated by \tref{t:nosmallingrp}, given that they have order
at least 21. 

Having completed the odd case, we may assume that $n$ is even. If
$n\equiv2\pmod4$ then $G$ misses $\ell=n$ by \tref{t:transgrp}. It
then follows from \tref{t:smallgrp} that $\Z_2$, $\Z_6$ and $D_6$
are near-omniversal but that no group of order $10$ is. Also for
$10<n\equiv2\pmod4$, \tref{t:nosmallingrp} rules out $\ell=(n+2)/2$,
showing that $G$ is not near-omniversal.

For $n\equiv4\pmod8$, \cyref{cy:halfn} rules out $\ell=n/2$. Hence
\tref{t:noabelpanmax} eliminates $n=4$ and
\tref{t:nosmallingrp} eliminates $12\le n\equiv4\pmod8$.

It remains to consider the case when $n$ is divisible by 8.
\tref{t:smallgrp} implies that $D_8$ is near-omniversal, but the
other groups of order 8 are not, given \tref{t:noabelpanmax} and
\cyref{cy:halfn}. Groups of order 16 or 24 are missing $\ell=\frac
n2+1$ by \tref{t:nosmallingrp}. Hence the abelian groups of order 16
or 24 are eliminated by \tref{t:noabelpanmax}. The non-abelian groups
of order 16 are near-omniversal, by \tref{t:smallgrp}. Also
$n\ne24$, because $\Z_3\rtimes\Z_8$ misses $\ell=12$ by
\cyref{cy:halfn} and all other non-abelian groups of order 24 miss
$\ell=15$, by \tref{t:smallgrp}.

Finally, all even groups of order $n>30$ are missing $\ell=\frac n2+1$
and $\ell=\frac n2+3$ by \tref{t:nosmallingrp}, so they are not
near-omniversal.
\end{proof}

\section{Concluding remarks}

In \tref{t:panall} we settled the existence question for omniversal
Latin squares. In \tref{t:nearpan2mod4} we showed that when omniversal
Latin squares do not exist, there are usually near-omniversal Latin
squares. We also showed in \tref{t:nosmallingrp} that group tables are
increasingly far from omniversal. One direction for future research is
raised by \tref{t:smallgrp}. It looks like groups (particularly those
of odd order) may miss some lengths not predicted by
\tref{t:nosmallingrp}. We showed that the $3n/5$ in that theorem is
best possible, but perhaps there are reasons why some greater lengths
are missed. Another question for Cayley tables is how much of a
subsquare must be present before we know that the whole subsquare is
present. A possible answer was proposed in \cjref{conj:extLSall}.

Although we proved in \tref{t:Near4n+2} that there exist near-omniversal
Latin squares of all orders $n \equiv 2 \pmod 4$, existence remains
open for all large orders $n\not\equiv 2 \pmod 4$. For any given
near-omniversal Latin square $L$ there is one length $\mu_L$ of
maximal partial transversal that is consistent with \eref{e:posslen}
but is not obtained in $L$.  It may be interesting to consider the set
of possible values for $\mu_L$ as a function of the order $n$ of
$L$. For odd orders, we know of no general restrictions on $\mu_L$.
However, for even orders, we can say more.  \tref{t:4n+2} implies
that for $n\equiv2\pmod4$ we have $\mu_L\in\{n/2,n\}$. For
$n\equiv0\pmod4$, \tref{t:everysecond} shows that it is
impossible to have $\mu_L\equiv n/2\pmod2$ and $n/2<\mu_L\le3n/4$.
This is because, if $\mu_L\ne n/2$ then there is a maximal partial
transversal of length $n/2$ and that can only be a transversal of
a subsquare of order $n/2$.

We now briefly examine how these restrictions compare with data for
Latin squares of small orders.  For this purpose, it is useful to
classify Latin squares by viewing them as sets of triples and then
using the natural action of $\sym_n\wr\sym_3$ on such sets. Orbits of
this action are called \emph{species} (also known as \emph{main
classes}). The action of $\sym_n\wr\sym_3$ preserves the lengths of
maximal partial transversals, so it suffices to study one representative
of each species.

For $n=6$, there are 10 near-omniversal species and both plausible
values for $\mu_L$ are realised (see \tref{t:nogrptran} and
\tref{t:Near4n+2}). There are also 2 species of Latin squares which
are not near-omniversal (since they are missing maximal partial
transversals of both lengths $n/2$ and $n$).  For $n=7$ all Latin
squares are either omniversal (91 species), near-omniversal with
$\mu_L=4$ (55 species), or isotopic to the Cayley table of $\Z_7$
(which only achieves maximal partial transversals of lengths 5 and 7).

For $n=8$, the above considerations show that $\mu_L\ne6$ but it
turns out that $\mu_L$ cannot be 7 either.
The overwhelming majority (283513 of the 283657 species) of
Latin squares are near-omniversal with $\mu_L=4$. There are 3 species
(including that of the dihedral group $D_8$) that are near-omniversal
with $\mu_L=5$, and 2 species that are near-omniversal with
$\mu_L=8$. One example of the latter type (with maximal partial
transversals of lengths 4, 5, 6, 7 displayed) is:
\[
\begin{tikzpicture}
\matrix[square matrix](M){
\mk{0}&1 &2 &3 &4 &5 &6 &7\\
1 &0 &\mk{3}&2 &5 &4 &7 &6\\
2 &3 &0 &\mk{1}&6 &7 &4 &5\\
3 &\mk{2}&1 &0 &7 &6 &5 &4\\
4 &5 &6 &7 &3 &2 &1 &0 \\
5 &4 &7 &6 &0 &1 &2 &3\\
6 &7 &4 &5 &2 &3 &0 &1\\
7 &6 &5 &4 &1 &0 &3 &2\\
};
\mvline{M}{5}
\mhline{M}{5}
\end{tikzpicture}
\;
\begin{tikzpicture}
\matrix[square matrix](M){
0 &\mk{1}&2 &3 &4 &5 &6 &7\\
1 &0 &3 &2 &5 &4 &\mk{7}&6\\
2 &3 &0 &1 &6 &7 &4 &5\\
3 &2 &1 &0 &7 &6 &5 &4\\
4 &5 &6 &7 &3 &2 &1 &\mk{0}\\
5 &4 &7 &6 &0 &1 &2 &3\\
\mk{6}&7 &4 &5 &2 &3 &0 &1\\
7 &6 &\mk{5}&4 &1 &0 &3 &2\\
};
\mvline{M}{5}
\mhline{M}{5}
\end{tikzpicture}
\;
\begin{tikzpicture}
\matrix[square matrix](M){
\mk{0}&1 &2 &3 &4 &5 &6 &7\\
1 &0 &3 &2 &5 &4 &7 &\mk{6}\\
2 &3 &0 &1 &6 &7 &\mk{4}&5\\
3 &2 &\mk{1}&0 &7 &6 &5 &4\\
4 &5 &6 &7 &3 &\mk{2}&1 &0 \\
5 &4 &7 &6 &0 &1 &2 &3\\
6 &7 &4 &\mk{5}&2 &3 &0 &1\\
7 &6 &5 &4 &1 &0 &3 &2\\
};
\mvline{M}{5}
\mhline{M}{5}
\end{tikzpicture}
\;
\begin{tikzpicture}
\matrix[square matrix](M){
\mk{0}&1 &2 &3 &4 &5 &6 &7\\
1 &0 &3 &2 &5 &4 &7 &\mk{6}\\
2 &3 &0 &1 &6 &7 &\mk{4}&5\\
3 &\mk{2}&1 &0 &7 &6 &5 &4\\
4 &5 &6 &7 &\mk{3}&2 &1 &0 \\
5 &4 &7 &6 &0 &\mk{1}&2 &3\\
6 &7 &4 &\mk{5}&2 &3 &0 &1\\
7 &6 &5 &4 &1 &0 &3 &2\\
};
\mvline{M}{5}
\mhline{M}{5}
\end{tikzpicture}
\]
For order $8$, there are 105 species of omniversal Latin squares, and 
34 species that are missing at least two lengths of maximal partial
transversals. Of those 34 species, only one is missing three lengths.
That species contains the following Latin square, which only has
maximal partial transversals of lengths 6 and 7 (as shown by the highlighted
cells)
\[
\begin{tikzpicture}
  \matrix[square matrix](M){
\mk{0}&1 &2 &3 &4 &5 &6 &7\\
1 &2 &3 &0 &5 &6 &7 &\mk{4}\\
2 &\mk{3}&0 &1 &6 &7 &4 &5\\
3 &0 &1 &2 &7 &4 &\mk{5}&6 \\
7 &6 &5 &4 &0 &3 &2 &1\\
6 &5 &4 &7 &3 &\mk{2}&1 &0\\
5 &4 &7 &{6}&2 &1 &0 &3\\
4 &7 &\mk{6} &5 &{1}&0 &3 &2\\
  };
\mvline{M}{5}
\mhline{M}{5}
\end{tikzpicture}
\qquad
\begin{tikzpicture}<
  \matrix[square matrix](M){
\mk{0}&1 &2 &3 &4 &5 &6 &7\\
1 &2 &3 &0 &5 &6 &7 &\mk{4}\\
2 &\mk{3}&0 &1 &6 &7 &4 &5\\
3 &0 &1 &2 &7 &4 &\mk{5}&6 \\
7 &6 &5 &4 &0 &3 &2 &1\\
6 &5 &4 &7 &3 &\mk{2}&1 &0\\
5 &4 &7 &\mk{6}&2 &1 &0 &3\\
4 &7 &{6} &5 &\mk{1}&0 &3 &2\\
  };
\mvline{M}{5}
\mhline{M}{5}
\end{tikzpicture}
\]

Examples like this prompt another research direction, which is to look
for Latin squares that have maximal partial transversals of as few
different lengths as possible.  Note that it follows from
\cite[Thm~13]{BMSW19} that almost all large Latin squares obtain only
sublinearly many lengths, and hence are very far from being
omniversal.  In contrast, \tref{t:everysecond} implies that the
presence of a maximal partial transversal of length $n/2$ in a Latin
square of (even) order $n$ necessitates the presence of maximal partial
transversals of linearly many lengths.

  \let\oldthebibliography=\thebibliography
  \let\endoldthebibliography=\endthebibliography
  \renewenvironment{thebibliography}[1]{
    \begin{oldthebibliography}{#1}
      \setlength{\parskip}{0.2ex}
      \setlength{\itemsep}{0.2ex}
  }
  {
    \end{oldthebibliography}
  }

\end{document}